\documentclass[11pt]{amsart}

\usepackage{amsmath,amsfonts,amsthm,euscript,amssymb}

\newtheorem{theorem}{Theorem}

\newtheorem{lemma}[theorem]{Lemma}

\newtheorem{corollary}[theorem]{Corollary}

\numberwithin{equation}{section}

\newcommand{\F}{\mathbb{F}}
\newcommand{\Q}{\mathbb{Q}}
\newcommand{\Z}{\mathbb{Z}}

\renewcommand{\O}{\EuScript{O}}

\newcommand{\GL}{\mathrm{GL}}
\newcommand{\GSp}{{\rm GSp}}

\newcommand{\Sp}{{\rm Sp}}

\newcommand{\Aut}{\mathrm{Aut}}
\newcommand{\Gal}{\mathrm{Gal}}
\newcommand{\Hom}{\mathrm{Hom}}

\newcommand{\Kbar}{\overline{K}}
\newcommand{\Aell}{A_\ell}

\newcommand{\Fellbar}{\overline{\F}_\ell}
\newcommand{\pr}[1]{\langle #1\rangle}

\newcommand{\rmk}{{\scshape Remark:~}}

\newcommand{\ideal}[1]{\mathfrak{{#1}}}
\newcommand{\p}{\mathfrak{p}}

\begin{document}

\title{An open image theorem for a general class of abelian varieties}
\author{Chris Hall}
\address{Department of Mathematics, University of Michigan at Ann Arbor}
\email{hallcj@umich.edu}
\subjclass[2000]{11G10,14K15}

\maketitle

\newcommand{\Neron}{N\'eron}

\begin{abstract}
\noindent
Let $K$ be a number field and $A/K$ be a polarized abelian variety with
absolutely trivial endomorphism ring.  We show that if the \Neron~model of
$A/K$ has at least one fiber with potential toric dimension one, then 
for almost all rational primes $\ell$, the Galois group of the splitting
field of the $\ell$-torsion of $A$ is $\GSp_{2g}(\Z/\ell)$.

\end{abstract}

\section{Introduction}

Let $K$ be a number field and $A/K$ a polarized abelian $g$-fold
with trivial $\Kbar$-endomorphism ring.  For each rational prime
$\ell$, let $\Aell$ denote the $\ell$-torsion of $A$ and $G_\ell$ the
Galois group of the splitting-field extension $K(\Aell)/K$.  If $g=1$, then
$A/K$ is an elliptic curve and a well-known theorem of Serre asserts that
$G_\ell$ is isomorphic to $\GSp_2(\Z/\ell)\simeq\GL_2(\Z/\ell)$ for all
sufficiently large $\ell$ \cite[Theorem 2]{S2}.  In a series of lectures and
letters Serre (cf. \cite[Corollaire au Th\'eor\`eme 3]{S4}) later showed how to extend the result to the
case when $g$ is odd, 2, or 6: if $\ell$ is sufficiently large, then
$G_\ell\simeq\GSp_{2g}(\Z/\ell)$.  However, for general $g$ it is an open
problem to show that $G_\ell$ is as big as possible for almost all $\ell$. 
In this paper we show that this is true when we assume an additional
hypothesis on the reduction of $A$, and our main theorem is the following.

\begin{theorem}
Suppose $A/K$ satisfies the following property:

\quad {\rm (T) : \begin{tabular}{p{4.5in}}
there is a finite extension $L/K$ so that the N\'eron model of $A/L$ over
the ring of integers $\O_L$ has a semistable fiber with toric-dimension one.
\end{tabular}}

\noindent
If $\ell$ is sufficiently large with respect to $A$ and $K$, then
$G_\ell\simeq \GSp_{2g}(\Z/\ell)$.  

\end{theorem}

The rest of this paper is devoted to a proof of the theorem.  In
section~\ref{sec::proof} we prove the theorem modulo a result of Serre on
the rigidity of inertial tori, and in section~\ref{sec::tori}
we prove the necessary rigidity result.

An example due to Mumford shows that one cannot remove hypothesis (T) from
the statement of our theorem \cite[Section 4]{M}.  More precisely, Mumford
constructed an abelian 4-fold with absolutely trivial endomorphism ring
such that $G_\ell$ does not contain $\Sp_{2g}(\Z/\ell)$ for infinitely many
$\ell$.  His 4-fold does not satisfy hypothesis (T) because it has
potentially good reduction everywhere; for infinitely many $\ell$ there are
no non-trivial unipotent elements $u\in G_\ell$ satisfying $(u-1)^2=0$, while
potential positive-dimension toric reduction would give rise to such
elements.  More importantly, Mumford's example fails to
satisfy the conclusion of the theorem because its so-called Mumford-Tate
group is strictly smaller than for those 4-folds addressed in the theorem
(cf.~\cite[annotation 4]{S3}).  However, the groups $G_\ell$ are as big as
possible for almost all $\ell$ once one takes into consideration the upper
bound imposed by the Mumford-Tate group, and for a general polarized
abelian $g$-fold $A/K$ it is conjectured that $G_\ell$ is almost always as
big as possible given the constraints imposed by the endomorphism ring and
Mumford-Tate group.  

\newcommand{\ord}{\mathrm{ord}}

While Mumford's example shows that an arbitrary abelian $g$-fold will not
satisfy the hypothesis (nor the conclusion) of the theorem, one can ask for
the likelihood that a ``random'' $A/K$ will have absolutely trivial
endomorphism ring and satisfy hypothesis (T).  For $g=1$, a necessary and
sufficient condition is that $j$-invariant does not lie in the ring of
integers of $K$.  For $g>1$, if $n=2g+2$ or $2g+1$ and $f(x)$ is a
degree-$n$ polynomial in $K[x]$ whose splitting field has Galois group
$S_n$, then Zarhin showed that the endomorphism ring of the Jacobian of the
hyperelliptic curve $y^2=f(x)$ is absolutely trivial \cite[Theorem 2.1]{Z}.  If moreover
there is a prime $\p$ in $K$ such that the reduction of $f(x)$ modulo $\p$
(is defined and) has $n-1$ distinct zeros (over an algebraic closure), one
of which is a double zero, then the Jacobian satisfies (T).

In an appendix to this paper E.~Kowalski shows that most monic polynomials in
$K[x]$ with integral coefficients satisfy both these properties, thus for
most hyperelliptic curves over $K$ the Jacobian satisfies the hypotheses of
the theorem.  Of course, for $g>2$ most polarized abelian $g$-folds $A/K$
do not arise as the Jacobians of hyperelliptic curves, so it is an open problem to
determine how often the hypotheses of the theorem are satisfied in general.



\subsection{Notation}

We use the notation $\ell\gg_{X} 0$ to mean that there is a constant
$\ell_0(X)$ which depends on the object $X$ and $\ell$ satisfies $\ell\geq
\ell_0(X)$.

\section{Proof of Main Theorem}\label{sec::proof}

Up to replacing $K$ by a finite extension $L/K$ we may assume $A/K$
satisfies (T) for $L=K$.  We fix an odd prime $\ell$ which is relatively prime
to the polarization degree of $A$.  We regard $V=\Aell$ as a vector space
over $\Z/\ell$ and write $\pr{\,,}:V\times V\to\mu_\ell$ for the Weil
pairing; the pairing exists because $A$ is polarized and it is
non-degenerate because $\ell$ is prime to the polarization degree.  If
$W\leq V$ is a subspace, then we write $W^\perp$ for the complement of $W$
with respect to $\pr{\,,}$.  We identify $\Gamma=\GSp(V)\leq\GL(V)$ with
the similitude subgroup of $\pr{\,,}$ and $\Sp(V)\leq\Gamma$ with the
isometry group.  There is short exact sequence
   $$ 1\to \Sp(V)\to \GSp(V)\overset{m}\to(\Z/\ell)^\times \to 1 $$
such that, for every $\gamma\in\Gamma$ and $x,y\in V$, we have
$\pr{\gamma x,\gamma y} =\pr{x,y}^{m(\gamma)}$.  The action of $G=G_\ell$
is compatible with $\pr{\,,}$, so there is an embedding $G\to\Gamma$ and
the theorem asserts it is an isomorphism if $\ell\gg_A 0$.  To prove the
theorem we will construct a subgroup $R\leq G$ which we can show
satisfies $R=\Sp(V)$ for $\ell\gg_A 0$, from which it will follow easily that
$G=\Gamma$ for $\ell\gg_{A,K} 0$.

\begin{lemma}
If $\ell\gg_A 0$, then $V$ is an irreducible $G$-module.
\end{lemma}

\begin{proof}
If $W\leq V$ is a $G$-submodule, then the isogeny $A\to B=A/W$ is defined
over $K$.  If $\phi_1,\phi_2:A\to B$ were isogenies for distinct $\ell$,
then $\psi=\phi_1\circ\phi_2^t$ would be an endomorphism outside of $\Z$
because $\deg(\psi)\not\in\deg(\Z)$, so distinct $\ell$ give rise to
distinct elements of the $K$-isogeny class of $A$.  However, Faltings'
theorem implies there can only be finitely many abelian varieties in the
$K$-isogeny class of $A$ (cf.~\cite[Corollaire 2.8]{D}), so there are only
finitely many $\ell$ such that $V$ is reducible.
\end{proof}

We say $\gamma\in\Gamma$ is a {\it transvection} if it is unipotent and
$V^{\gamma=1}$ has codimension one.  The $\Gamma$-conjugate of a
transvection is a transvection, and we write $R\leq G$ for the normal
subgroup generated by the subset of transvections in $G$.  The proof of the
following lemma
shows that condition (T) is sufficient, but not necessary, to show that $R$
is non-trivial for almost all $\ell$.

\begin{lemma}
If $\ell\gg_A 0$, then $R$ is non-trivial.
\end{lemma}

\begin{proof}
Suppose $\p$ is a prime in $\O_K$ over which $A$ has toric-dimension
one; $\p$ exists because $A$ satisfies (T) for $L=K$.  Then the monodromy
about $\p$ is a transvection provided $\ell$ does not divide the order of
the component group of the N\'eron model of $A$ over $\p$ (cf.~2.1, 2.5
and 3.5 of \cite{G}).  
\end{proof}

\newcommand{\G}{\mathbb{G}}
\newcommand{\R}{\EuScript{R}}

Suppose $R$ is non-trivial.  Let $W\leq V$ be a non-trivial irreducible
$R$-submodule and $H\leq G$ be the stabilizer of $W$. 

\begin{lemma}
If $V$ is an irreducible $G$-module, then $W\cap W^\perp=W^R=0$ and
$V=\bigoplus_{G/H} gW$.  
\end{lemma}

\begin{proof}
This follows from (the proof of) lemma 3.2 of \cite{H} because
$(gW)^\perp=g(W^\perp)$ for all $g\in\Gamma$.
\end{proof}

Suppose $V$ is an irreducible $G$-module, and let $R_g\leq R$ be the
subgroup generated by the transvections which act non-trivially on $gW$. 
The image of $R_g=gR_1g^{-1}$ in $\Sp(gW)$ is non-trivial, irreducible, and
generated by transvections, so is all of $\Sp(gW)$ by \cite[Main Theorem]{ZS}.

\begin{lemma}
If $g_1H\neq g_2H$ as cosets, then the commutator $[R_{g_1},R_{g_2}]$ is trivial.
\end{lemma}

\begin{proof}
Note, $g_1W=g_2W$ if and only if $g_1H=g_2H$.  If $\gamma\in R$ is a
transvection which acts non-trivially on $g_1W$, then $(\gamma-1)V\leq
g_1W$.  Thus if $g_1W\neq g_2W$, then $(\gamma-1)g_2W$ lies in $g_1W\cap
g_2W=0$, so $R_{g_1}$ acts trivially on $g_2W$.  In particular, if
$\gamma_1\in R_{g_1}$, $\gamma_2\in R_{g_2}$ are transvections and
$g_1W\neq g_2W$, then for each $gW$, at least one of $\gamma_1,\gamma_2$
acts trivially on $gW$, so the restrictions of $\gamma_1,\gamma_2$ to $gW$
commute (for every coset $gH$), hence they commute on all of $V$.
\end{proof}

The lemma implies $R_g=\Sp(gW)$ for all $gH$ and $R$ is the central product
$\prod_{G/H} R_g$.  Therefore, if we write $n=[V:W]$, then $N_\Gamma(R)$ is
isomorphic to the wreath product $\GSp(W)\wr S_n$ and $G\leq N_\Gamma(R)$.
The next step is to show that $n=1$.

Let $N_0\leq N_\Gamma(R)$ be the kernel of $N_\Gamma(R)\to S_n$.


\begin{lemma}\label{lemma6}
Let $\pi$ be a prime in $\O_K$ and let $e$ be the ramification index of
$\pi$ over $\Q$.  If $\ell>en+1$ and $I\leq G$ is the inertia subgroup
of a prime in $K(A_\ell)$ over $\pi$, then the image of $I$ in $S_n$
is trivial.
\end{lemma}

\begin{proof}
If $\ell>n$, then $S_n$ has no elements of order $\ell$, so the image in
$S_n$ of the $\ell$-Sylow subgroup of $I$ must be trivial.  In particular,
if $\pi$ does not lie over $\ell$, then $I$ is an $\ell$-group because it
is trivial or generated by a unipotent element, so the image of $I$ in $S_n$
is trivial.  Therefore we may suppose $\pi$ lies over $\ell$.  Let $C\leq
I$ be a complement of the $\ell$-Sylow subgroup; it is a cyclic subgroup of
order prime to $\ell$.  By section 1.13 of \cite{S1} (following \cite{R}),
there is a finite extension $\F_{\ell^d}/\F_\ell$ and a surjective
homomorphism $T=\F_{\ell^d}^\times\to C$ so that the representation
$T\to\GL(V)$ has amplitude at most $e$; see section~\ref{sec::tori} for the
definition of amplitude.  The kernel of $T\to S_n$ has index at most $n$ in
$T$ and it commutes with $Z(N_0)$ (because it lies in $N_0$), so
lemma~\ref{lemma::rig} of section~\ref{sec::tori} implies all of $T$
commutes with $Z(N_0)$ because $\ell>en+1$.  In particular, the centralizer
of $Z(N_0)$ in $N_\Gamma(R)$ is $N_0$, so $C$, the image of $T\to\GL(V)$,
and $I$ lie in $N_0$.  
\end{proof}

The lemma implies that the fixed field of the kernel of $G\to S_n$
is unramified and has uniformly bounded degree over $K$.  By a theorem of
Hermite, there are only finitely many such extensions, so up to replacing
$K$ by a finite extension we may assume that the image of $G$ in $S_n$ is
trivial (for all $\ell\gg_A 0$).  Therefore $G\leq N_0=\prod_{G/H} \GSp(gW)$
and hence $n=1$ because $G$ acts irreducibly, so $W=V$ and $R=\Sp(V)$.  Once
we know that $R$ is big, the following lemma completes the proof of the
theorem.

\begin{lemma}
If $R=\Sp(V)$ and $\ell\gg_K 0$, then $G=\Gamma$.
\end{lemma}

\begin{proof}
If $\ell\gg_K 0$, then $K$ and $\Q(\mu_\ell)$ are disjoint extensions of
$\Q$.  On the other hand, if $R=\Sp(V)$, then $G/R$ is the Galois group of
$K(\mu_\ell)/K$, so must be all of $(\Z/\ell)^\times$ for $\ell\gg_K 0$.
\end{proof}

\rmk Most of the above carries through if we replace $K$ by a global field of
characteristic $p>0$.  One key difference is that $G_\ell$ is no longer equal to
$\GSp_{2g}(\Z/\ell)$ for all $\ell\gg_A 0$, but it does contain
$\Sp_{2g}(\Z/\ell)$ for all $\ell\gg_A 0$.  Another difference is that the argument in lemma~\ref{lemma6} is made simpler
by the fact that there are no inertial tori to contend with for $\ell\neq p$.


\section{Rigidity of Tori}\label{sec::tori}

\newcommand{\Tbar}{\underline{T}}
\newcommand{\rhobar}{\underline{\rho}}
\newcommand{\GLbar}{\underline{\GL}}
\newcommand{\amp}{\mathrm{amp}}

Let $T$ be the multiplicative group of a finite extension
$\F_{\ell^d}/\F_\ell$.  We regard $T$ as the set of $\F_\ell$-points
of the algebraic torus $\Tbar/\F_\ell$ given by the Weil restriction
of scalars of the split one-dimensional torus $\G_m/\F_{\ell^d}$.
The $d$ fundamental characters $\psi_1,\ldots,\psi_d:\Tbar\to\G_m$
corresponding to the $d$ embeddings $\F_{\ell^d}\to\Fellbar$ form a basis
for the character group $\Hom_{\,\Fellbar}(\Tbar,\G_m)$.  We define the
amplitude of a character $\chi=\prod_i \psi_i^{e_i}$ as $\max_i(e_i)$,
and we say $\chi$ is $\ell$-restricted if $0\leq e_i\leq \ell-1$ for
all $i$ and $e_i<\ell-1$ for some $i$ (cf.~\cite[section 1.7]{S1}
and \cite[annotation 5]{S4}).

Let $\F_\lambda/\F_\ell$ be a finite extension and $V/\F_\lambda$
be a finite-dimensional vector space.  We say a representation
$\rho:\Tbar\to\GLbar(V)$ of algebraic groups over $\F_\lambda$ is
$\ell$-restricted if all its characters are $\ell$-restricted.  Every
representation $\rho:T\to\GL(V)$ extends uniquely to an $\ell$-restricted
representation $\rho:\Tbar\to\GLbar(V)$, and we define the amplitude of
$\rho$ as the maximum of the amplitudes of its characters.

\begin{lemma}
    Let $s\in\GL(V)$ be a semisimple element and $\rho:T\to\GL(V)$ be a
    representation of amplitude $e<\ell-1$.  If $\rho(T)$ commutes with
    $s$, then $\rho(\Tbar)$ commutes with $s$ in $\GLbar(V)$.
\end{lemma}

\begin{proof}
Up to a base change by a finite extension of $\F_\lambda$, we may assume
that $s$ is diagonalizable in $\GL(V)$.  Thus there is a decomposition
$V=\oplus_j V_j$ so that $s$ acts on $V_j$ via an element $s_j\in
Z(\GL(V_j))$ and $T$ preserves the decomposition because it commutes
with $s$.  The amplitude of the restriction $\rho_j:T\to\GL(V_j)$
is at most $e$.  If we write $\rho_j:\Tbar\to\GLbar(V_j)$ for the
$\ell$-restricted representation corresponding to $\rho_j$, then
$\rho_j(\Tbar)$ commutes with $s_j$ because $s_j$ is a scalar.  In
particular, the composition of the product representation
$\Tbar\to\prod_j\GLbar(V_j)$ with the obvious embedding
$\prod_j\GLbar(V_j)\to\GLbar(V)$ is an $\ell$-restricted representation
extending $\rho:T\to\GL(V)$, so it must be $\rho:\Tbar\to\GLbar(V)$ and
hence $\rho(\Tbar)$ commutes with $s$.  
\end{proof}

The power of the previous lemma is that it allows us to show that
representations $\rho:T\to\GL(V)$ are `rigid' if they have sufficiently
small amplitude (cf.~\cite{S4}).

\begin{lemma}\label{lemma::rig}
    Let $s\in\GL(V)$ be a semisimple element and $\rho:T\to\GL(V)$ be a
    representation of amplitude $e$.  If $S\leq T$ is a subgroup such that
    $\rho(S)$ commutes with $s$ in $\GL(V)$ and $e\cdot [T:S]<\ell-1$,
    then $\rho(T)$ commutes with $s$ in $\GLbar(V)$.
\end{lemma}

\begin{proof}
Let $c=[T:S]$ and $\rho^c:T\to\GL(V)$ denote the composition of $\rho$
with the $c$th-power-map $c:T\to T$.  By assumption $\rho^c$ has
amplitude $e\cdot c<\ell-1$, hence the corresponding $\ell$-restricted
representation $\rho^c:\Tbar\to\GLbar(V)$ is the composition of the
$\ell$-restricted representation $\rho:\Tbar\to\GLbar(V)$ with the
$c$th-power-map $c:\Tbar\to\Tbar$.  Moreover, $\rho^c(T)\leq\rho(S)$
commutes with $s$ in $\GL(V)$, so by the previous lemma $\rho^c(\Tbar)$
commutes with $s$ in $\GLbar(V)$.  In particular, $c:\Tbar\to\Tbar$
is surjective because $0<c<\ell-1$, hence $\rho(\Tbar)$ and a foriori
$\rho(T)$ commute with $s$.
\end{proof}

\section{Acknowledgements}

We gratefully acknowledge helpful conversations with Serre, and in
particular, for clarifications with regard to initial tori.  We also
acknowledge helpful conversations with N.M.~Katz and E.~Kowalski.


\section*{Appendix: Most hyperelliptic curves have big monodromy}

{\centerline{Emmanuel Kowalski\footnote{ ETH Z\"urich - DMATH, {\tt kowalski@math.ethz.ch}}}}

\bigskip

\newtheorem{proposition}[theorem]{Proposition}

\theoremstyle{remark}
\newtheorem{remark}[theorem]{Remark}

\def\sumb{\mathop{\sum \Bigl.^{\flat}}\limits}
\newcommand{\mods}[1]{\,(\mathrm{mod}\,{#1})}
\newcommand{\Fp}{\mathbf{F}}
\newcommand{\ra}{\rightarrow}

Let $k/\Q$ be a number field and $\Z_k$ its ring of integers. Let
$f\in \Z_k[X]$ be a monic squarefree polynomial of degree $n=2g+2$ or
$2g+1$ for some integer $g\geq 1$, and let $C_f/k$ be the (smooth,
projective) hyperelliptic curve of genus $g$ with affine equation
$$
C_f\,:\, y^2=f(x),
$$
and $J_f$ its jacobian.
\par
In the previous text, C. Hall has shown that the image of the Galois
representation
$$
\rho_{f,\ell}\,:\, \Gal(\bar{k}/k)\ra
\Aut(J_f[\ell](\bar{k}))\simeq \Fp_{\ell}^{2g}
$$
on the $\ell$-torsion points of $J_f$ is as big as possible for almost
all primes $\ell$, if the following two (sufficient) conditions hold:
\par
(1) the endomorphism ring of $J_f$ is $\Z$;
\par
(2) for some prime ideal $\p\subset \Z_k$, the fiber over
$\p$ of the N\'eron model of $C_f$ is a smooth curve except for
a single ordinary double point.
\par
These conditions can be translated concretely in terms of the
polynomial $f$, and are implied by:
\par
(1') the Galois group of the splitting field of $f$ is the full
symmetric group $\mathfrak{S}_n$ (this is due to a result of
Zarhin~\cite{zarhin}, which shows that this condition implies (1));
\par
(2') for some prime ideal $\p\subset \Z_k$, $f$ factors in
$\Fp_{\p}=\Z_k/\p\Z_k$ as
$f=f_1f_2$ where $f_i\in \Fp_{\p}[X]$ are relatively prime
polynomials such that
$f_1=(X-\alpha)^2$ for some $\alpha\in \Fp_{\p}$ and $f_2$ is
squarefree of degree $n-2$; indeed, this implies (2).
\par
In this note, we show that, in some sense, ``most'' polynomials $f$
satisfy these two conditions, hence ``most'' jacobians of
hyperelliptic curves have maximal monodromy modulo all but finitely
many primes (which may, a priori, depend on the polynomial, of
course!).
\par
More precisely, for $k$ and $\Z_k$ as above, let us denote
$$
\mathcal{F}_n=\{f\in \Z_k[X]\,\mid\, f\text{ is monic of degree $n$} \},
$$
and let the height be defined on $\mathcal{F}_n$ by
$$
H(a_0+a_1X+\cdots +a_{n-1}X^{n-1}+X^n)=
\max_{0\leq i\leq n-1}{H(a_i)},
$$
where $N$ is the norm from $k$ to $\Q$ and $H$ is any reasonable
height function on $k$, e.g., choose a $\Z$-basis $(\omega_i)_{1\leq
  i\leq d}$ of $\Z_k$, where $d=[k:\Q]$, and let
$$
H(\alpha_1\omega_1+\cdots +\alpha_d\omega_d)=\max |\alpha_i|,
$$
for all $(\alpha_i)\in\Z^d$. 
\par
Let $\mathcal{F}_n(T)$ denote the finite set
\begin{equation}\label{eq-base}
\mathcal{F}_n(T)=\{f\in \mathcal{F}_n\,\mid\, H(f)\leq T\}.
\end{equation}
\par
We have $|\mathcal{F}_n(T)|=N_k(T)^n$, 
where
$$
N_k(T)=|\{x\in \Z_k\,\mid\, H(x)\leq T\}|\asymp T^d,\text{ where }
d=[k:\Q]. 
$$
\par
Say that $f$ has {\it big monodromy} if the Galois group of its splitting
field is $\mathfrak{S}_n$.  We will show:

\begin{proposition}\label{pr-1}
Let $k$ and $\Z_k$ be as above. Then
\begin{gather*}
  |\{f\in\mathcal{F}_n(T)\,\mid\, \text{$f$ does not have big monodromy} \}|
      \ll N_k(T)^{n-1/2}(\log N_k(T)),
\end{gather*}
for all $T\geq 2$, where the implied constant depends on $k$ and $n$.
\end{proposition}

Say that $f\in\mathcal{F}_n$ has \emph{ordinary ramification} if it
satisfies condition (2') above.

\begin{proposition}\label{pr-2}
  Let $k$ and $\Z_k$ be as above, and assume $n\geq 2$.
  There exists a constant $c>0$,
  depending on $n$ and $k$, such that we have
$$
|\{f\in\mathcal{F}_n(T)\,\mid\, 
\text{$f$ does not have ordinary ramification}\}|
\ll \frac{N_k(T)^{n}}{(\log N_k(T))^{c}}
$$
for $T\geq 3$, where the implied constant depends on $k$ and $n$.
\end{proposition}

Finally, say that $J_f$ has {\it big monodromy} if the image of
$\rho_{f,\ell}$ is as big as possible for almost all primes $\ell$.

\begin{corollary}
  Assume that $n\geq 2$. Then we have
$$
\lim_{T\ra +\infty}{
\frac{1}{|\mathcal{F}_n(T)|}{
|\{
f\in\mathcal{F}_n(T)\,\mid\,
J_f\text{ does not have big monodromy}
\}|
}
}=0.
$$
\end{corollary}

\begin{remark}
  Quantitatively, we have proved that the rate of decay of this
  probability is at least a small power of power of logarithm, because
  of Proposition~\ref{pr-2}.  With more work, one should be able to
  get $c$ equal or very close to $1$, but it seems hard to do better
  with the current ideas (the problem being in part that we must avoid
  $f$ for which the discriminant is a unit in $\Z_k$, which may well
  exist, and sieve can not detect them better than it does
  discriminants which generate prime ideals, the density of which
  could be expected to be about $(\log N_k(T))^{-1}$).
\end{remark}


For both propositions, in the language of~\cite{lsieve}, we consider a
sieve with data
\begin{gather*}
(\mathcal{F}_n,\{\text{prime ideals in $\Z_k$}\},
\{\text{reduction modulo  $\p$}\}),\quad
(\mathcal{F}_n(T),\text{counting measure}),
\end{gather*}
and we claim that the ``large sieve constant'' $\Delta$ for the
sifting range
$$
\mathcal{L}^*=\{\p\subset \Z_k\,\mid\, N\p\leq L\}
$$
satisfies
$$
\Delta\ll N_k(T)^{n}+L^{2n},
$$
where the implied constant depends only on $k$. Indeed, this follows
from the work of Huxley~\cite{huxley}, by combining in an obvious
manner his Theorem 2 (which is the case $n=1$, $k$ arbitrary) with his
Theorem 1 (which is the case $k=\Q$, $n$ arbitrary).
\par
Concretely, this implies that for arbitrary subsets
$\Omega_{\p}$ in the image of $\mathcal{F}_n$ under reduction
modulo $\p$ --- the latter is simply the set of monic polynomials of
degree $n$ in $\Fp_{\p}[X]$, and has cardinality
$(N\p)^n$ --- we have
\begin{equation}\label{eq-ls1}
|\{f\in\mathcal{F}(T)\,\mid\, f\mods{\p}\notin
\Omega_{\p}\text{ for } N\p\leq L\}|
\ll (N_k(T)^{n}+L^{2n})\Bigl(
\sumb_{N\ideal{a}\leq L}{
\prod_{\p\mid \ideal{a}}
{\frac{|\Omega_{\p}|}{(N\p)^n-|\Omega_{\p}|}
}}
\Bigr)^{-1},
\end{equation}
where the sum is over squarefree ideals in $\Z_k$ with norm at most
$L$, and therefore also
\begin{equation}\label{eq-ls}
|\{f\in\mathcal{F}(T)\,\mid\, f\mods{\p}\notin
\Omega_{\p}\text{ for } N\p\leq L\}|
\ll (N_k(T)^{n}+L^{2n})\Bigl(
\sum_{N\p\leq L}{
\frac{|\Omega_{\p}|}{(N\p)^n}
}
\Bigr)^{-1}.
\end{equation}

Proposition~\ref{pr-1} is a result of S.D. Cohen~\cite{cohen}; it is
also a simple application of the methods of Gallagher~\cite{gallagher}
(one only needs~(\ref{eq-ls}) here), the basic idea being that
elements of the Galois group of the splitting field of a polynomial
$f$ are detected using the factorization of $f$ modulo prime
ideals. We recall that the first quantitative result of this type (for
$k=\Q$) is due to van der Waerden~\cite{vdW}, whose weaker result
would be sufficient here (though the proof is not simpler than
Gallagher's). 

\begin{proof}[Proof of Proposition~\ref{pr-2}]
Let $\p\subset \Z_k$ be a prime ideal, and let
$\Omega_{\p}$ be the set of polynomials
$f\in\Fp_{\p}[X]$ which are monic of degree $n$ and factor as
described in Condition (2'). We claim that, for some constant $c>0$,
$c\leq 1$ (depending on $k$ and $n$), we have
\begin{equation}\label{eq-lower}
\frac{|\Omega_{\p}|}{(N\p)^n}\geq \frac{c}{N\p}
\end{equation}
for all prime ideals with norm $N\p\geq P_0$, for some $P_0$
depending on $k$ and $n$.
\par
Indeed, for $n\geq 4$, we have clearly
$$
|\Omega_{\p}|\geq (N\p)\times 
|\{
f\in \Fp_{\p}[X]\,\mid\, \deg(f)=n-2,\ \text{$f$ monic irreducible}
\}|;
$$
for $n=2$, this holds with the convention that $1$ is
irreducible of degree $0$, and for $n=3$, we must subtract 1 from the second
term on the right. If $n=2$, we are done, otherwise it is
well-known that
$$
|\{
f\in \Fp_{q}[X]\,\mid\, \deg(f)=n-2,\ \text{$f$ monic irreducible}
\}|\sim \frac{q^{n-2}}{n-2}
$$
as $q\ra +\infty$, hence the lower bound~(\ref{eq-lower}) follows by
combining these two facts (showing we can take for $c$ any constant
$<(n-2)^{-1}$ if $P_0$ is chosen large enough; using more complicated
factorizations of the squarefree factor of degree $n-2$, one could get
$c$ arbitrarily close to $1$).
\par
Now we apply~(\ref{eq-ls}) with this choice of subsets for $\p$
with norm $>P_0$, and with $\Omega_{\p}=\emptyset$ for other
$\p$. We take $L=N_k(T)^{d/2}$, assuming that $L>P_0$, i.e.,
that $T$ is large enough. Since, if $f\in\mathcal{F}_n(T)$ does not
have ordinary ramification, we have by definition
$f\mods{\p}\notin \Omega_{\p}$ for any $\p$, it
follows by simple computations that
$$
|\{f\in\mathcal{F}_n(T)\,\mid\,\text{$f$ does not have ordinary
  ramification} \}|
\ll N_k(T)^{n}
H^{-1}
$$
where the implied constant depends on $k$ and
$$
H=\sumb_{N\ideal{a}\leq L}{
c^{\omega(\ideal{a})}(N\ideal{a})^{-1}},
$$
where now $\sumb$ restricts the sum to squarefree ideals not divisible
by a prime ideal of norm $\leq P_0$, and where $\omega(\ideal{a})$ is
the number of prime ideals dividing $\ideal{a}$. 
\par
It is then a standard fact about sums of multiplicative functions that 
$$
H\gg (\log L)^{c}
$$
for $L$ large enough (depending on $P_0$; recall that $0<c\leq 1$),
and this leads to the proposition, since $L$ and $N_k(T)$ are
comparable in logarithmic scale.
\end{proof}


\begin{thebibliography}{XXXX99a}

\bibitem[D]{D} P.~Deligne, ``Preuve des conjectures de Tate et de
Shafarevitch (d'apr\`es G.~Faltings),'' Seminar Bourbaki, Vol.~1983/84, {\it
Ast\'erisque} No.~121-122 (1985), 25--41.

\bibitem[G]{G} A.~Grothendieck, ``Mod\`eles de N\'eron et monodromie'', SGA 7 Part I, expos\'e IX, Springer Lecture Notes in Mathematics, Vol. 288, 1972.

\bibitem[H]{H} C.~Hall, ``Big symplectic or orthogonal monodromy modulo $\ell$,'' Duke Math.~Journal 141 (2008), 179--203.

\bibitem[M]{M} D.~Mumford, ``A note of Shimura's paper ``Discontinuous groups
and abelian varieties'','' {\it Math.~Ann.~}181 (1969), 345--351.

\bibitem[R]{R} M.~Raynaud, ``Sch\'emas en groupes de type $(p,\ldots, p)$,'' Bull.~Soc.~Math.~France 102 (1974), 241--280.

\bibitem[S1]{S1} J-P Serre, ``Propri\'et\'es galoisiennes des points d'ordre fini des courbes elliptiques,'' Invent.~Math.~15 (1972), no.~4, 259--331.

\bibitem[S2]{S2} J-P Serre, {\it Abelian $\ell$-adic Representations and
Elliptic Curves}, 2nd ed., Adv.~Book Classics, Addison-Wesley, Redwood City,
Calif., 1989.

\bibitem[S3]{S3} J-P Serre, letter to Daniel Bertrand, 8/6/1984, {\it Collected Papers}, Vol. 4.

\bibitem[S4]{S4} J-P Serre, Lettre \`a Marie-France Vign\'eras du 10/2/1986, {\it Collected Papers}, Vol. 4.


\bibitem[ZS]{ZS} A.E.~Zalesski\u\i, V.N. Sere\v zkin, ``Linear groups
    generated by transvections,'' (Russian) Izv. Akad. Nauk SSSR Ser. Mat.
    40 (1976), no. 1, 26--49, 221; translation in Math. USSR Izvestija,
    Vol. 10 (1976), no. 1, 25--46.

\bibitem[Z]{Z} Y.G.~Zarhin, ``Hyperelliptic Jacobians without complex
multiplication,'' Math.~Res.~Lett. 7 (2000), no.~1, 123--132.

\end{thebibliography}

\begin{thebibliography}{CCC}

\bibitem[C]{cohen}
S.D. Cohen: \textit{The distribution of the Galois groups of integral
  polynomials}, Illinois J. Math. 23 (1979), 135--152.

\bibitem[G]{gallagher}
P.X. Gallagher: \textit{The large sieve and probabilistic Galois
  theory}, in Proc. Sympos. Pure Math., Vol. XXIV, Amer. Math. Soc.
(1973), 91--101.

\bibitem[Hu]{huxley}
M.N. Huxley: \textit{The large sieve inequality for algebraic number
  fields}, Mathematika 15 (1968) 178--187.

\bibitem[K1]{lsieve}
E. Kowalski: \textit{The large sieve and its applications: arithmetic
  geometry, random walks, discrete groups}, Cambridge Univ. Tracts (to
appear).
    
\bibitem[vdW]{vdW} B.L.~van der Waerden: \textit{Die Seltenheit der
    reduziblen Gleichungen und der Gleichungen mit Affekt},
  Monath.~Math.~Phys.~43 (1936), 133--147.

\bibitem[Z]{zarhin} Y.G.~Zarhin: \textit{Hyperelliptic Jacobians
    without complex multiplication}, Math.~Res.~Lett. 7 (2000), no.~1,
  123--132.

\end{thebibliography}
\end{document}